\documentclass[letterpaper, 10 pt, conference]{ieeeconf}  
\usepackage{amsfonts}
\usepackage{graphicx}
\usepackage{float}
\usepackage{amsthm}
\usepackage{amsmath}
\usepackage{color}

\usepackage{url}
\usepackage{tabulary}
\usepackage{booktabs}
\usepackage[dvipsnames,table]{xcolor}

\usepackage[ruled, lined, linesnumbered, commentsnumbered, longend]{algorithm2e}
\DeclareMathOperator*{\argmin}{arg\,min}
\SetKwProg{Init}{Initialization}{}{}

\IEEEoverridecommandlockouts                              

\newtheorem{problem}{Problem}
\newtheorem{theorem}{Theorem}
\newtheorem*{remark}{Remark}

\title{\LARGE \bf
Learning the Delay Using Neural Delay Differential Equations} 

\author{Maria Oprea$^{1}$, Mark Walth$^{2}$, Robert Stephany$^{1}$, Gabriella Torres Nothaft$^{2}$,  Arnaldo Rodriguez-Gonzalez$^{3}$,\\ and William Clark$^{2}$
\thanks{*This work was funded by NSF grant DMS-1645643.}
\thanks{$^{1}$M. Oprea and R. Stephany are with the Center for Applied Mathematics, Cornell University, Ithaca NY, {\tt\small \{mao237, rrs254\}@cornell.edu}}
\thanks{$^{2}$M. Walth, G. Torres Nothaft, and W. Clark are with the Department of Mathematics, Cornell University, Ithaca NY, {\tt\small \{msw283,gt285,wac76\}@cornell.edu}}
\thanks{$^{3}$A. Rodriguez-Gonzalez is with the Sibley School of Mechanical and Aerospace Engineering, Cornell University, Ithaca NY, {\tt\small ajr295@cornell.edu}}
}

\begin{document}

\maketitle
\thispagestyle{empty}
\pagestyle{empty}

\begin{abstract}

The intersection of machine learning and dynamical systems has generated considerable interest recently.  Neural Ordinary Differential Equations (NODEs) represent a rich overlap between these fields. In this paper, we develop a continuous time neural network approach based on Delay Differential Equations (DDEs). Our model uses the adjoint sensitivity method to learn the model parameters and delay directly from data. Our approach is inspired by that of NODEs and extends earlier neural DDE models, which have assumed that the value of the delay is known a priori. We perform a sensitivity analysis on our proposed approach and demonstrate its ability to learn DDE parameters from benchmark systems. We conclude our discussion with potential future directions and applications.

\end{abstract}

\section{INTRODUCTION}

Neural ordinary differential equations (NODEs) have proven to be an efficient framework for solving various problems in machine learning \cite{https://doi.org/10.48550/arxiv.1806.07366,dupont_augmented_2019}. 
NODEs were inspired by the observation that a traditional residual neural network can be viewed as a discretized solution to an ordinary differential equation (ODE), where each ``layer" corresponds to a single time step in the discretization, and the  ``depth" corresponds to the length of integration time. 
NODEs begin with this implied ODE model but treat the solution to this ODE as the hidden state.
In this framework, a forward pass through the network is computed by calling an ODE solver, and backpropagation works by integrating an appropriate adjoint equation backward in time. 
NODEs are a time and memory-efficient model for various regression and classification tasks \cite{https://doi.org/10.48550/arxiv.1806.07366,dupont_augmented_2019}.

In this work, we develop a novel continuous-time neural network approach based on delay differential equations (DDEs). 
The mathematical structure of delay differential equations differs substantially from that of ODEs \cite{Hale1993IntroductionTF}. 
As a result, phenomena modeled by DDEs are often poorly represented by ODE models. 
There has been recent progress in developing Neural DDEs (NDDEs), the DDE counterpart of NODEs \cite{anumasa2021delay}, \cite{Zhu_ICLR_2021}, \cite{zhu_neural_2022}.
For example, \cite{Zhu_ICLR_2021} and \cite{zhu_neural_2022} show that NDDEs can learn models from time series data which cannot be learned by standard NODEs. 

\begin{figure}[H]
    \centering
    \includegraphics[width = 0.5\textwidth]{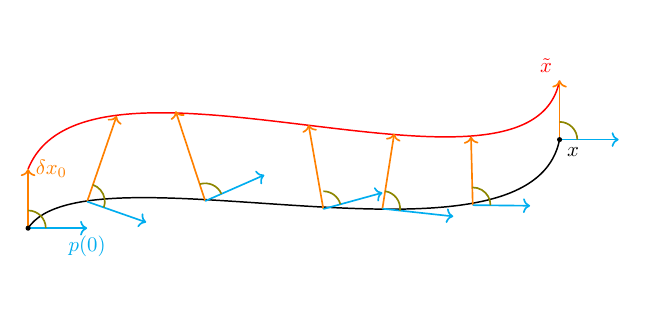}
    \caption{Relation of adjoint to the variational equation. The black curve is the reference trajectory; red curve is a nearby perturbed trajectory. Orange represents the variation between the red curve and black curve. The blue arrows are the adjoint.  Notice that the angle between the adjoint and the variation is constant.}
    \label{fig:adjoint_perturbation}
\end{figure}

\noindent They also show that NDDEs can perform classification tasks that standard NODEs cannot.

Our work extends \cite{anumasa2021delay}, \cite{Zhu_ICLR_2021} and \cite{zhu_neural_2022} in several ways. 
The authors of \cite{Zhu_ICLR_2021} assume the magnitude delay is known a priori, and their model cannot learn the delay as a parameter. 
By contrast, by performing a sensitivity analysis of our NDDE model, we derive an adjoint equation that allows our model to learn the delay from data. 
This difference makes our approach more applicable as a system identification tool, as the exact value of the delay is often unknown in practice. 

There are myriad examples of dynamic processes whose evolution depends not only on the system's current state but also on the state at a previous time;
such systems can be modeled by delay differential equations. 
Examples of such systems are wide-ranging: in computer networks, delays arise due to the transmission time of information packets \cite{yu_lmi_2004}; in gene expression dynamics, a delay occurs due to the time taken for the messenger RNA to copy genetic code and transport macromolecules from the nucleus to the cytoplasm \cite{busenberg_interaction_1985}; in population dynamics, delay enters due to the time taken for a species to reach reproductive maturity \cite{kuang_delay_1993}. 
Further, in engineering applications, delays can be used as a control mechanism: such applications have arisen in the theory of machine tool vibrations \cite{kalmar-nagy_subcritical_2001}, automotive powertrains \cite{jankovic_developments_2009}, gantry crane oscillations \cite{nayfeh_delayed-position_2005}, and many others.

The main contribution of this paper can be summarized as follows: we propose a novel algorithm for learning the delay and parameters of an unknown underlying DDE based on observations of trajectory data. 
Our approach is based on the adjoint sensitivity method (see Fig. \ref{fig:adjoint}), which we describe in detail in Sec. \ref{sec:adjoint}. 
We implement the algorithm summarized in Sec. \ref{sec:algo} using \texttt{PyTorch}, and provide the code used in this work at the end of that section.
We then establish the efficacy of our approach on demonstrative examples in Sec. \ref{sec:example}.

\begin{figure*}[h!]
    \centering
    \includegraphics[width = \textwidth]{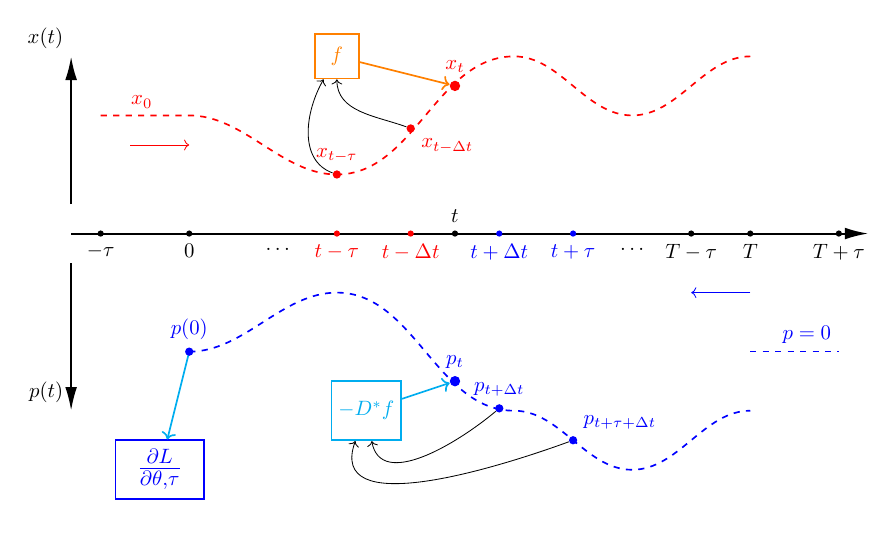}
    \caption{Schematic drawing of the adjoint method for computing gradients. The dashed red curve on the upper half of the figure represents the computed solution to the original delay differential equation $\dot{x}(t)=f(x(t),x(t-\tau))$. We use constant initial condition $x(t)=x_0$ for $t\in[-\tau,0]$. The dashed blue curve on the bottom half represents the computed solution to the adjoint problem, where the adjoint operator $-D^*f$ takes present and past values of the computed adjoint solution $p$ to calculate future values of $p$.}
    \label{fig:adjoint}
\end{figure*}

\section{Problem statement}

A DDE with a single delay, $\tau > 0$, can be written in the form:
\begin{equation}\label{DDE}
    \dot{x}(t) = f(x(t), x(t - \tau))
\end{equation}
where $f$ is a functional that gives tangents of the curve $x(t) \in Q$.
Here, $Q$ is a Euclidean space.
The DDE's initial conditions are a function $x_0:[-\tau, 0]\rightarrow Q$. 
In general, $f:Q\times Q \rightarrow TQ$ can be any function that takes in two elements ($x$ and $y$) from the manifold $Q$ and gives an element of the tangent space of the first entry: $f(x, y) \in T_xQ$. 
Throughout this paper, however, we will assume that $f$ belongs to a known parameterized family, $\{ f_\theta : Q \times Q \to TQ \mid \theta \in \mathbb{R}^d \}$.
The goal is to learn the parameter $\theta$ corresponding to $f$. 

Let $L:Q\times \mathbb{R} \rightarrow \mathbb{R}^+$ be a loss function of the form:
$$L(x(\cdot)) = \int_0^T \ell(x(t)) \, dt + g(x(T))$$ where $\ell : Q \rightarrow \mathbb{R}$ is the running cost and $g:Q \rightarrow\mathbb{R}$ is the terminal cost. 

We are interested in the following scenario: We have a data set consisting of $N_{data}$ measurements, $\{ X_j \}_{j = 0}^{N_{data} - 1}$, taken from the flow of an unknown delay differential equation.
We assume this DDE has a constant initial condition given by the first measurement point $X_0$, i.e. that $x(t)=X_0$ for $t\in[-\tau,0]$.

The main goal of this paper is to solve the following minimization problem:
\begin{problem}
Find 
\begin{alignat*}{3}
    \argmin_{\tau > 0,\ \theta \in \mathbb{R}^d} \quad& L(x(\cdot&&)) && \\
    \ subject \ to \quad& \dot{x}(t)\ &&= f_\theta \big( x(t), x(t - \tau) \big),  &&\quad\, 0 < t \leq T\\
    & x(t) &&= X_0, &&-\tau \leq t\leq 0
\end{alignat*}
\end{problem}

\section{Proposed algorithm for finding the delay}\label{sec:algo}

The overarching idea is to find the minimizer of the loss function using the gradient descent method. 
To accomplish this, we use the adjoint sensitivity method to obtain the gradient of the loss with respect to the model parameters, $\theta$, and the delay $\tau$.  

Throughout this paper, we will assume our physical space is $\mathbb{R}^n$, with the usual metric and inner product. 
The cotangent space is also isomorphic to $\mathbb{R}^n$, which means that we can think of the adjoint as a row vector in $\mathbb{R}^n$. 

\subsection{Adjoint sensitivity method}\label{sec:adjoint}
To derive our result, we use an analysis that is similar to the one for ODEs in \cite{where_adjoint_comes_from}. 
Consider a general DDE as in \eqref{DDE}. 
In the following, we write: 
$$y(t) = \begin{cases}  x(t - \tau) ,  \text{ if } t > \tau \\
                        x_0(t), \quad\ \text{ if } t \leq 0\end{cases}$$
and treat $y$ as an independent variable.

The goal is to obtain a formula for the sensitivity of the trajectory at a given time, $t$, with respect to variations in the initial conditions. 
The following idea comes from the calculus of variations and is at the core of mechanics and continuous optimization \cite{abraham_marsden}.
Let $V$ be the vector space of all small perturbations, $\delta x_0$, to the initial condition.
These perturbations will propagate forward through the functional, $f$, engendering a perturbed trajectory $\Tilde{x}(t)=  x(t) + \delta x(t)$. 
We consider a family of perturbations $\{ \delta_\varepsilon x_0 \}$, with corresponding perturbed trajectories $\{ \delta_\varepsilon x(t) \}$ at time $t$, for which $\lim_{\varepsilon \to 0} \delta_\varepsilon x(t) = 0$.
The evolution of the perturbation is described by the equation $\dot{\delta_\varepsilon x} = \dot{\Tilde{x}} - f(x, y)$. 
To compute the sensitivity, one would need to solve a nonlinear differential equation for $\Tilde{x}$ from $0$ to $t$ and then observe $\delta_{\varepsilon} x(t) / \varepsilon$ as $\varepsilon \to 0$. 
The idea of the adjoint method is that instead of working in $V$, we consider its dual space $V^*$. 
In particular, we are looking for the adjoint $p \in V^*$ such that  $p(\delta_\varepsilon x)$ stays constant as in Fig. \ref{fig:adjoint_perturbation}.
The advantage of using the dual is that the adjoint satisfies a linear equation. 
Moreover, we can obtain a formula for the sensitivity directly from the values of $p$, without needing to compute the perturbed trajectories and take the limit as $\varepsilon \to 0$.

In $\mathbb{R}^n$, the perturbations are column vectors, while the adjoints are row vectors. 
We will denote the pairing between an element and its dual by $\langle \cdot, \cdot\rangle$.

We can obtain the formula for the time evolution of $p$ by defining the action for the DDE $\dot{x}(t) = f(x(t),\ x(t - \tau)) = f(x, y)$ with initial condition $x_0:[-\tau, 0]\rightarrow \mathbb{R}$ as:
$$ S = \int_{0}^T \langle p(t),\dot{x}(t)\rangle - H(x(t), y(t), p(t))\,dt $$
where $H(x, y, p) = \langle p, f(x, y)\rangle$ is the Hamiltonian. 

The true path $(x(t), p(t))$ is a critical value of this action. 
With this knowledge, we can find the equation of motion for the adjoint $p$ using the calculus of variations. 
We vary $x$ and $p$, keeping the initial $x_0$ and the final $x(T)$ fixed. We consider perturbations $\delta_\varepsilon x = \varepsilon\delta x$ for some fixed $\delta x$. 
\begin{multline*}
    \delta S =\lim_{\varepsilon \to 0 } \frac{1}{\varepsilon}(S(p + \varepsilon\delta p, x + \varepsilon\delta x) - S(p, x)) = \\ \langle p(T),\ \delta x(T) \rangle - \langle p(0),\ \delta x_0 \rangle  + \int_{0}^{T} \langle \delta p,\ \dot{x} - f(x, y) \rangle - \\ \langle \dot p(t) + Df^*_x p(t) +  Df_y^* p(t + \tau)\boldsymbol{1}_{t < T - \tau},\ \delta x(t) \rangle \, dt
\end{multline*}

The boundary terms vanish since the variations have fixed endpoints. 
Setting $\delta p \to 0$ recovers the initial DDE while letting $\delta x \to 0$ recovers the adjoint equation:
\begin{align}\label{eq:DDE_adjoint}
    \dot{p}(t) &= -Df_x^*\big( p(t) \big) - Df^*_y \big( p(t + \tau) \big)\boldsymbol{1}_{t < T - \tau}(t), 
\end{align}
where $\boldsymbol{1}_{t<T-\tau}$ is the indicator on the set $t<T-\tau$. Here, $Df_y$ is the partial derivative of $f$ with respect to its first variable, $y$. 
In $\mathbb{R}^n$, $Df_y$ is the $n\times n$ matrix with entries: $(Df_y)_{ij} = \frac{\partial f^i}{\partial x^j}$. 
The star denotes the adjoint with respect to the inner product. 
That is, 
$$\langle p, Df_x \dot x  \rangle = \langle Df_x^*(p),  \dot x  \rangle.$$ 
In $\mathbb{R}^n$ with the usual inner product, the adjoint is the matrix transpose $Df_x^* = Df_x^T$. 
Similarly, $Df_y$ is the partial derivative of $f$ with respect to the second component, which is the $n \times n$ matrix $(Df_y)_{ij} = \frac{\partial f^i}{\partial y_j}$. 

Throughout this paper, we will use superscripts to index the components of the vectors, $x$, and subscripts for the components of the adjoint covectors, $p$. 

Defining the adjoint in this way is particularly useful because it allows us to efficiently compute the sensitivity of the trajectory with respect to parameters.
The following theorem summarizes this fact:

\begin{theorem}
Let $x(t)$ denote the solution of the DDE \eqref{DDE} with constant initial function $x_0(t) = x_0 \in \mathbb{R}^n$. For each $i \in \{1, \dots, n\}$ suppose the adjoint $p_i$ satisfies the delay differential equations given in \eqref{eq:DDE_adjoint}, with terminal condition: $p_i(T) = e_i$ (where $e_i$ denotes the $i$th standard basis vector) and $p_i(t) = 0 , t > T$. Then

\begin{align*}
    \frac{\partial x^i(T)}{\partial x_0} &= p_i(0), \\
    \frac{\partial x^i(T)}{\partial \theta} & = \int_0^T p_i(t) \frac{\partial f}{\partial \theta} \big( x(t), y(t) \big) dt, \\
    \frac{\partial x^i(T)}{\partial \tau} &= - \int_0^{T -\tau} p_i(t + \tau)\frac{\partial f}{\partial y} \big( x(t + \tau), x(t) \big) \dot x(t)dt.
\end{align*}
\end{theorem}

\begin{proof}
As in \cite{where_adjoint_comes_from}, we compute variations of $\dot{x} = f(x, y, \theta)$ and integrate them against the adjoint. 
After doing integration by parts and plugging in the formula for the adjoint equation \eqref{eq:DDE_adjoint} we obtain
\begin{multline*}
     p(T) \delta x(T) = p(0) \delta x(0) + \delta \theta \int_0^T p_i(t) \frac{\partial f}{\partial \theta} \big( x(t), y(t) \big) dt - \\ \delta \tau\int_0^{T -\tau} p_i(t + \tau)\frac{\partial f}{\partial y} \big( x(t + \tau), x(t) \big)\dot x(t)dt. 
\end{multline*}
If we set the the boundary condition $p_i(T) = e_i$ then on the left-hand side we are left with the $i$th component of $\delta x$.
On the right-hand side, the limit as $\delta \theta$, $\delta \tau $ and $\delta x(0)$ respectively go to $0$ gives us the desired formulas for the sensitivity.
\qed
\end{proof}

$~$

Using this result, we can derive expressions for the gradient of the loss with respect to the parameters. 
We distinguish two particular cases
\subsubsection{No running cost}

If $\ell = 0$, we can directly find the gradients using the chain rule:
\begin{align*} 
    \frac{\partial L}{\partial \theta} &= \sum_{i = 1}^n \frac{\partial g}{\partial x^i}\Big|_{x(T)}\int_0^T p_i(s) \frac{\partial f}{\partial \theta}\big( x(s), y( s) \big)ds,
\end{align*}
$~$
\begin{multline*}
    \frac{\partial L }{\partial \tau} = \sum_{i = 1}^n \frac{\partial g}{\partial x^i}\Big|_{x(T)}\int_0^{T-\tau} p_i(s + \tau)  \\ 
    \frac{\partial f}{\partial y}\big( x(s + \tau), y(s + \tau) \big)  \dot x (s) ds.
\end{multline*}

\begin{remark}
We can treat the case where
\begin{align}\label{eq:loss:discrete}
    L(x, \tau, \theta) &= \sum_{j = 0}^{N_{data} - 1}\|X_{j} - x(t_{j}) \|^2 \nonumber \\ &= \sum_{j = 0}^{N_{data} - 1} \sum_{i = 1}^n \left(X^i_{j} - x^i(t_{j})\right)^2,
\end{align}
in the same manner. 
In this case, the gradients with respect to $\theta$ and $\tau$ are:
\begin{align}\label{eq:loss_gradient:discrete}
    \frac{\partial L}{\partial \theta, \tau} = -2 \sum_{j = 0}^{N_{data} - 1} \sum_{i = 1}^n (X^i_j - x^i(t_j))\frac{\partial x^i(t_j)}{\partial \theta, \tau}.
\end{align}

Since the adjoint equation has the same terminal condition independent of the time, we do not need to solve a different DDE for each $j$. 
Suppose we have the adjoint trajectory for $x^i(T)$, denoted by $p_i$. 
Then the adjoint for $x^i(t_j)$ is 
$$p_i^j(t) = p_i(T - t_j + t),\ t \in [0, t_j].$$
Therefore:
\begin{align*}
    \frac{\partial x^i(t_j)}{\partial \theta} & = \int_0^{t_j} p_i(T - t_j + t) \frac{\partial f}{\partial \theta} (x(t), y(t)) dt, 
\end{align*}
\begin{multline*}
    \frac{\partial x^i(t_j)}{\partial \tau}  = - \int_0^{t_j -\tau} p_i(T - t_j + t + \tau) \\ \frac{\partial f}{\partial y}(x(t + \tau), x(t)) \dot x(t)dt.
\end{multline*}

\end{remark}
\subsubsection{Fully general case}
We define the extended action $$\Tilde{S} = \int_0^T \ell(x(t)) + \langle p(t),\dot{x}\rangle - H(x(t), y(t), p(t)) dt. $$

As before, we vary $x$ and $p$ with fixed endpoints to find the extended adjoint equation:
\begin{equation}\label{eq:extended_adjoint}
        \dot{p}(t) = -Df_q^*(p(t)) - Df^*_y(p(t + \tau))\boldsymbol{1}_{t < T - \tau}(t) + \frac{\partial l}{\partial x}.
\end{equation}

Let $p(t)$ be the solution to \eqref{eq:extended_adjoint} with terminal condition $p(T) = -\frac{\partial g}{\partial x}\Big|_{x(T)}$, $p(t) = 0 , t > T$. 
Then the gradients of the loss are:
\begin{align*}
    \frac{\partial L}{\partial \theta} &= -\int_0^T p(t) \frac{\partial f}{\partial \theta} \big(x(t), y(t)\big) dt,  \\
    \frac{\partial L}{\partial \tau} &=  \int_0^{T -\tau} p(t + \tau) \frac{\partial f}{\partial y}\big(x(t + \tau), x(t)\big)\dot x(t)dt.
\end{align*}
A schematic drawing of the procedure is found in Fig. \ref{fig:adjoint_perturbation}.

\subsection{Proposed algorithm for finding the delay}
We assume the target data $\{ X_j \}_{j = 0}^{N_{data} - 1}$ represents measurements of a single DDE trajectory. 
Since we are working with discrete measurements, we use the loss function given in \eqref{eq:loss:discrete}.
Using the formulas for the gradients given above, we train a model to learn the delay and parameters of the DDE that the data comes from.

Starting from a random initial guess for $\theta$ and $\tau$, we compute the predicted trajectory of the DDE with these parameters. 
Once we have this trajectory, we solve \eqref{eq:DDE_adjoint} backward in time to find the adjoint. 
From there, we use  \eqref{eq:loss_gradient:discrete} to compute the gradient of the loss. 
Finally, we perform a step of gradient descent using the \texttt{Adam} optimizer \cite{kingma2014adam}. 
We repeat this process for $N_{epochs}$ iterations, or until the target and predicted trajectories are sufficiently close. 
Algorithm \ref{alg:two} summarizes our approach.

We implemented Algorithm \ref{alg:two} in \texttt{PyTorch}.
For our implementation, we wrapped the forward and backward procedures in a \texttt{torch.autograd.Function} class.
In the forward pass, we use a DDE solver to calculate the predicted trajectory.
For the experiments in this paper, we used a first-order Euler step for our DDE solver.
In the backward pass, we use the same DDE solver to calculate the adjoint by solving \eqref{eq:DDE_adjoint} backward in time.
Our code then computes and returns the gradient of the loss with respect to $\tau$ and $\theta$. 
Our implementation is open source and is available at: \url{https://github.com/punkduckable/NDDE}.

\begin{algorithm}[hbt!]
\caption{Learning delay and parameters from data}\label{alg:two}
\KwData{$\{ X_j \}_{j = 0}^{N_{data} - 1}$}
\KwResult{$\theta, \tau$}
\Init{$\theta, \tau \sim U(-2, 2)$  }{ }
\For {$i \gets 1, \ldots, N_{epochs}$}{
Solve $\ x'(t) = f(x(t), x(t - \tau))$, $\ x|_{[-\tau, 0]} = X_0$\;
Compute $L(x(t), \theta, \tau)$\;
Check for convergence\;
Solve $\ \dot{p}(t) = -p(t)\frac{\partial f}{\partial q} - p(t + \tau) \, \frac{\partial f}{\partial y}\, \boldsymbol{1}_{t < T - \tau}(t) $, $p|_{[T, T + \tau]} = 0$, $p(T) = 1$\;
Compute $\frac{\partial L}{\partial \theta}$ and $\frac{\partial L}{\partial \tau}$\;
Update $\theta$ and $\tau$
 }
\end{algorithm}

\section{Examples}\label{sec:example}

In this section, we test the algorithm outlined in section \ref{sec:algo} on two simple examples. 
Specifically, we learn the parameters and the delay in a logistic delay equation model, and an delay exponential delay model. 
For these experiments, we use the loss given by \eqref{eq:loss:discrete} and update $\theta$ and $\tau$ using equations \eqref{eq:loss_gradient:discrete}.
To generate the training sets, we use $T = 10$ then numerically solve each DDE using the forward Euler method with a step size of $dt = 0.1$. 
We use the resulting discretized solution (with $100$ data points) the target values, $\{ X_j \}_{j = 1}^{N_{Data} - 1} \}$.
Further, in both experiments, we use the \texttt{Adam} optimizer with a learning rate of $0.1$.

$~$

\noindent\textbf{Logistic Delay Equation:} 
First, we study the logistic delay equation. 
This equation was first proposed to understand oscillatory phenomena in ecology and is used to model the dynamics of a single population growing toward a saturation level with a constant reproduction rate \cite{delay_logistic}.

For this experiment the true DDE is:
$$\dot x(t) = x(t)\left(1 - x(t - \tau)\right),$$ 
with constant initial condition $x_0 = 2$. 
We use our algorithm on this equation.
For this experiment, we use the model: 
$$f(x, y, \theta) = \theta_0 x(1 - \theta_1 y).$$ 

We initialize our model with $\theta_0 = \theta_1 = 1.75$ and $\tau = 1.75$.
We then run our algorithm until the loss reaches a threshold of $0.001$. 
Table \ref{tab:Logistic} reports the discovered values of $\theta_0$, $\theta_1$, and $\tau$. 
Thus, our implementation identifies the correct parameters and $\tau$ values.

\begin{table}[]
    \centering
    \caption{results for the delay logistic equation with $N_{epochs}=144$.}
    \rowcolors{2}{white}{olive!20!green!15}

    \begin{tabulary}{\linewidth}{C C C}
        \toprule[0.3ex]
        &  True parameters &  Discovered parameters \\
        \midrule[0.1ex]
         $\theta_0$ & $1$ & $1.00325$ \\
         $\theta_1$ & $1$ &  $1.00232$ \\ 
         $\tau $ & $1$ & $0.99347$ \\
         \bottomrule[0.3ex]    \end{tabulary}
        
    \label{tab:Logistic}
\end{table}

Fig. \ref{fig:Loss_Logistic_3D} and Fig. \ref{fig:Loss_Logistic_2D} show that the loss function has a sink around the true parameters and $\tau$ values.
Further, Fig. \ref{fig:Loss_Logistic_2D} shows the path 
our algorithm takes as it converges to the true values.

\begin{figure}
    \centering
    \includegraphics[width = 0.48\textwidth]{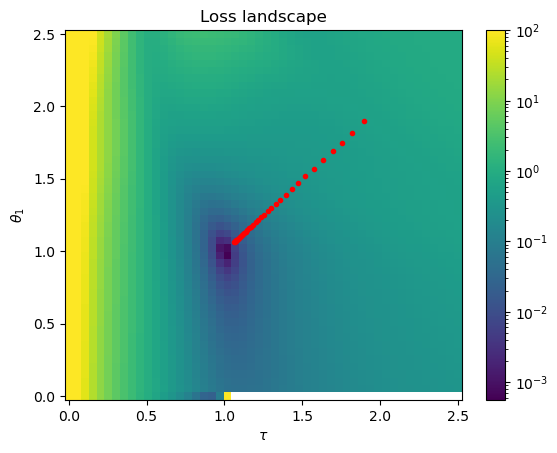}
    \caption{Parameter exploration of the loss landscape for various values of $\theta_1$ and $\tau$; darker colors indicate smaller values of the loss function. Red dots represent the learned parameters after each iteration of Algorithm \ref{alg:two}.}
    \label{fig:Loss_Logistic_3D}
\end{figure}

\begin{figure}
    \centering
    \includegraphics[width = 0.48\textwidth]{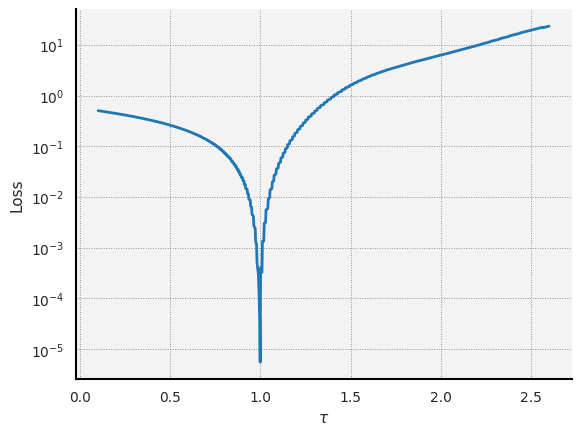}
    \caption{A semilogarithmic plot of the loss function in the Logistic Delay equation experiment for various values of $\tau $ and fixed $\theta_{1, 2} = 1$.}
    \label{fig:Loss_Logistic_2D}
\end{figure}

Fig. \ref{fig:Traj_Logistic} shows that the target and final predicted trajectories are closely aligned.

\begin{figure}
    \centering
    \includegraphics[width = 0.48\textwidth]{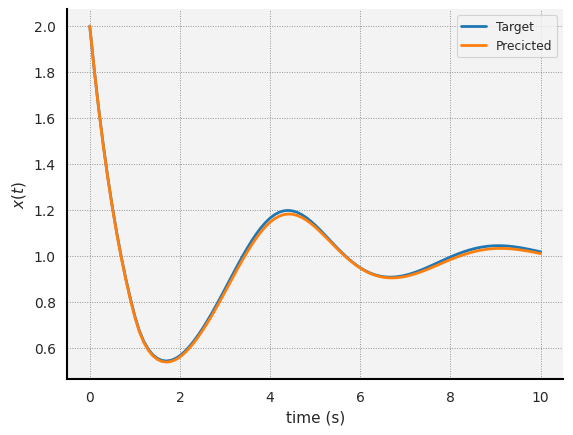}
    \caption{Trajectories of the final predicted (orange) vs. true (blue) trajectories for the delay logistic equation.}
    \label{fig:Traj_Logistic}
\end{figure}

\begin{table}[]
    \centering
    \caption{results for the delay exponential decay equation with $N_{epochs}=505$.}
    \rowcolors{2}{white}{olive!20!green!15}
        
    \begin{tabulary}{\linewidth}{C C C}
        \toprule[.3ex]
         & True parameters &  Discovered parameters \\
        \midrule[.1ex]
        $\theta_0$ & $-2$ & $-2.02428$ \\
        $\theta_1$ &  $-2$ &  $-2.01691$ \\ 
        $\tau $ & $1$ & $0.99085$ \\
        \bottomrule[.3ex]
    \end{tabulary}

    \label{tab:Exponential}
\end{table}

$~$

\noindent\textbf{Delay Exponential Decay Equation: } 
Second, we consider the delay exponential decay equation. 
This equation arises in the study of cell growth, specifically in the case when the cells need reach maturity before reproducing \cite{thesis_cell_growth}. 

For this experiment, the true DDE is: 
$$\dot x(t) = -2x(t) - 2x(t - 1),$$ 
with the constant initial condition $x_0 = -1$. 
For this experiment, we use the model:
$$f(x, y, \theta) = \theta_0 x + \theta_1 y.$$ 

We initialize our model with $\theta_0 = \theta_1 = -3.0$ and $\tau = 2.0$.
We then run our algorithm until the loss reaches a threshold of $0.001$. 
Table \ref{tab:Exponential} reports the discovered values of $\theta_0$, $\theta_1$, and $\tau$. 
Fig. \ref{fig:Traj_Exponential} shows the true and learned trajectories.

 \begin{figure}
    \centering
    \includegraphics[width=0.48\textwidth]{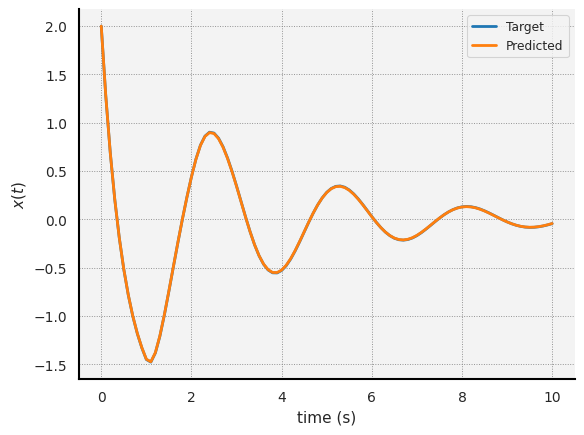}
    \caption{Trajectories of the final predicted (orange) vs. target (blue) trajectories for the delay exponential decay equation.}
    \label{fig:Traj_Exponential}
\end{figure}

\section{CONCLUSIONS}

In this paper, we expanded the neural delay differential equation framework by presenting an algorithm to learn the time delay directly from data. 
In practice, the delay is rarely known a priori.
Therefore, our advance is crucial for studying real-world systems governed by delay differential equations. 

We defined the adjoint as the critical path of the action and deduced that it must satisfy a linear DDE. 
We then showed that solving this DDE and integrating the adjoint trajectory against the partials of the functional $f$ allows us to compute the gradient of the loss with respect to the delay and the parameters.
We demonstrated the efficacy of our approach by applying our algorithm to the delay logistic and the delay exponential decay.

There are many potential future directions of research to extend our results.
First, from a theoretical standpoint, we would like to extend our analysis of the adjoint sensitivity for DDEs whose state space is a nonlinear manifold. 
Second, from a numerical perspective, we would like to make our algorithm robust enough to tackle real-world data. 
To begin, we need to make sure that our algorithm is robust with respect to noise and small fluctuations in the coefficients. 
For this purpose, using a naive Euler step to find the DDE solutions is likely insufficient.
Even for well-behaved DDEs, such as the examples we considered, the Euler step has stability issues if the parameters get too large or too small. 
Further research into the robustness of DDE solvers and the interplay between the solver and the adjoint is needed.
Additionally, to demonstrate the practical applicability of our approach, we would like to demonstrate that we can learn a good DDE model from sparse or noisy data, since applications often do not allow for collecting data quite as cleanly as we have used in our experiments.

Lastly, in the limit as the delay goes to zero, we observe that the trajectory values blow up in finite time.
Moreover, DDEs with small delays are challenging to tackle numerically, as most Runge-Kutta methods are stable only if the step size is smaller than the delay \cite{INTHOUT1996237}. 
Algorithms to bypass this limitation exist but are often computationally expensive \cite{SHAMPINE2001441}.
Therefore, we would like to come up with an analytical expression for the adjoint equations in the limit as $\tau \to 0$ and use it to build an algorithm that can distinguish between a true DDE with positive $\tau$ and an ODE with $\tau = 0$.

\addtolength{\textheight}{-12cm}   





\section*{ACKNOWLEDGMENT}
We would like to thank Maximilian Ruth for insightful and constructive conversations on this topic.



\bibliographystyle{plain}
\bibliography{bibliography.bib}

\begin{thebibliography}{10}

\bibitem{abraham_marsden}
R.~Abraham and J.E. Marsden.
\newblock {\em Foundations of {Mechanics}}.
\newblock Addison-Wesley Publishing Company Inc., 1985.

\bibitem{anumasa2021delay}
S.~Anumasa and P.K. Srijith.
\newblock Delay {Differential Neural Networks}.
\newblock In {\em 2021 6th International Conference on Machine Learning
  Technologies}, pages 117--121, 2021.

\bibitem{busenberg_interaction_1985}
S.~Busenberg and J.~Mahaffy.
\newblock Interaction of {Spatial} {Diffusion} and {Delays} in {Models} of
  {Genetic} {Control} by {Repression}.
\newblock {\em Journal of Mathematical Biology}, 22(3):313--333, 1985.

\bibitem{https://doi.org/10.48550/arxiv.1806.07366}
R.T.Q. Chen, Y.~Rubanova, J.~Bettencourt, and D.~Duvenaud.
\newblock Neural {Ordinary Differential Equations}, 2018.

\bibitem{dupont_augmented_2019}
E.~Dupont, A.~Doucet, and Y.W. Teh.
\newblock Augmented {Neural} {ODEs}.
\newblock In {\em Advances in {Neural} {Information} {Processing} {Systems}},
  volume~32. Curran Associates, Inc., 2019.

\bibitem{Hale1993IntroductionTF}
J.K. Hale and S.~V. Lunel.
\newblock Introduction to {Functional Differential Equations}.
\newblock In {\em Applied Mathematical Sciences}, 1993.

\bibitem{INTHOUT1996237}
K.J. in~'t Hout.
\newblock On the {Stability} of {Adaptations} of {Runge-Kutta Methods} to
  {Systems} of {Delay Differential Equations}.
\newblock {\em Applied Numerical Mathematics}, 22(1):237--250, 1996.
\newblock Special Issue Celebrating the Centenary of Runge-Kutta Methods.

\bibitem{jankovic_developments_2009}
M.~Jankovic and I.~Kolmanovsky.
\newblock Developments in {Control} of {Time}-{Delay} {Systems} for
  {Automotive} {Powertrain} {Applications}.
\newblock In {\em Delay {Differential} {Equations}: {Recent} {Advances} and
  {New} {Directions}}, pages 55--92. Springer US, Boston, MA, 2009.

\bibitem{kalmar-nagy_subcritical_2001}
T.~Kalmár-Nagy, G.~Stépán, and F.C. Moon.
\newblock Subcritical {Hopf} {Bifurcation} in the {Delay} {Equation} {Model}
  for {Machine} {Tool} {Vibrations}.
\newblock {\em Nonlinear Dynamics}, 26(2):121--142, October 2001.

\bibitem{kingma2014adam}
D.P. Kingma and J.~Ba.
\newblock Adam: A {Method for Stochastic Optimization}.
\newblock {\em arXiv preprint arXiv:1412.6980}, 2014.

\bibitem{kuang_delay_1993}
Y.~Kuang.
\newblock Delay {Differential} {Equation} with {Application} in {Population}
  {Dynamics}.
\newblock January 1993.

\bibitem{nayfeh_delayed-position_2005}
A.H. Nayfeh, Z.N. Masoud, and N.A. Nayfeh.
\newblock A {Delayed}-{Position} {Feedback} {Controller} for {Cranes}.
\newblock In G.~Rega and F.~Vestroni, editors, {\em {IUTAM} {Symposium} on
  {Chaotic} {Dynamics} and {Control} of {Systems} and {Processes} in
  {Mechanics}}, Solid {Mechanics} and its {Applications}, pages 385--395,
  Dordrecht, 2005. Springer Netherlands.

\bibitem{where_adjoint_comes_from}
F.A. Rihan.
\newblock {\em Delay {Differential} {Equations} and {Applications} to
  {Biology}}.
\newblock Springer, 2021.

\bibitem{thesis_cell_growth}
F.A. Rihan and G.A. Bocharov.
\newblock Numerical {Modelling} in {Biosciences} using {Delay} {Differential}
  {Equations}.
\newblock {\em Journal of Computational and Applied Mathematics},
  125:183–199, 200.

\bibitem{delay_logistic}
B.~Ruth and R.~Gergely.
\newblock Global {Dynamics} of a {Novel Delayed Logistic Equation Arising} from
  {Cell Biology}.
\newblock {\em Journal of Nonlinear Science}, 30:397--418, 2020.

\bibitem{SHAMPINE2001441}
L.F. Shampine and S.~Thompson.
\newblock Solving {DDEs} in {Matlab}.
\newblock {\em Applied Numerical Mathematics}, 37(4):441--458, 2001.

\bibitem{yu_lmi_2004}
M.~Yu, L.~Wang, T.~Chu, and F.~Hao.
\newblock An {LMI} {Approach} to {Networked Control Systems with Data Packet
  Dropout and Transmission Delays}.
\newblock In {\em 2004 43rd {IEEE} {Conference} on {Decision} and {Control}
  ({CDC})}, volume~4, pages 3545--3550 Vol.4, December 2004.

\bibitem{Zhu_ICLR_2021}
Q.~Zhu, Y.~Guo, and W.~Lin.
\newblock Neural {Delay Differential Equations}.
\newblock 02 2021.

\bibitem{zhu_neural_2022}
Q.~Zhu, Y.~Shen, D.~Li, and W.~Lin.
\newblock Neural {Piecewise}-{Constant} {Delay} {Differential} {Equations}.
\newblock {\em Proceedings of the AAAI Conference on Artificial Intelligence},
  36(8):9242--9250, June 2022.

\end{thebibliography}

\end{document}